\title{A Glimpse of Arithmetic Dynamics}
\author{Ryan Grady\thanks{
                  Montana State University}
        \and
        Mark Poston\thanks{
                      Montana State University. Mark Poston was supported during this research by a grant from the Undergraduate Scholars Program at Montana State.
                  }
        }
\documentclass{article}
\usepackage{graphicx}
\usepackage{tkz-berge,wrapfig,mdframed, multicol}

\usepackage{amsmath}
\usepackage{amsfonts}
\usepackage{amssymb}
\usepackage{amsthm}
\usepackage{comment}
\usepackage{epsfig}
\usepackage{psfrag}
\usepackage{mathrsfs}
\usepackage{amscd}
\usepackage[all]{xy}
\usepackage{rotating}
\usepackage{lscape}
\usepackage{amsbsy}
\usepackage{verbatim}
\usepackage{mathtools}
\usepackage{moreverb}
\usepackage{graphicx}
\usepackage{hyperref}
\usepackage{slashed}

\usepackage[latin1]{inputenc}
\usepackage{tikz}
\usetikzlibrary{trees}
\usetikzlibrary[shapes]
\usetikzlibrary[arrows]
\usetikzlibrary{patterns}
\usetikzlibrary{fadings}
\usetikzlibrary{backgrounds}
\usetikzlibrary{decorations.pathreplacing}
\usetikzlibrary{decorations.pathmorphing}
\usetikzlibrary{decorations.footprints}
\usetikzlibrary{decorations.markings}
\usetikzlibrary{shapes,arrows}

\usetikzlibrary{positioning}

\def\L8{L_\infty}

\def\CC{\mathbb C}
\def\FF{\mathbb F}\def\GG{\mathbb G}\def\HH{\mathbb H}

\def\NN{\mathbb N}
\def\RR{\mathbb R}

\def\ZZ{\mathbb Z}

\def\fJ(E){\mathfrak E}
\def\fJ{\mathfrak J}

\newtheorem{theorem}{Theorem}
\newtheorem*{fermat}{Fermat's Little Theorem}
\newtheorem{lemma}{Lemma}
\newtheorem{proposition}{Proposition}
\newtheorem{definition}{Definition}
\newtheorem{example}{Example}

\begin{document}
\newpage
\maketitle




In this note, we offer a palatable introduction to the field of {\it arithmetic dynamics}. That is, we study the patterns that arise when iterating a polynomial map.  This note is accessible to those who have taken an introductory proof based course and some linear algebra; the appendix utilizes abstract algebra. A more sophisticated overview of the field is given in \cite{dynamics}.

In our study, we consider polynomials defined not over fields like the real or complex numbers, but rather fields which have only finitely many elements. These finite fields arise in abstract algebra and number theory, and they are used extensively in the mathematics underlying cryptography. We pay particular attention to the shape of the patterns arising from iterating a given polynomial, so called orbit types.

We begin by introducing (discrete) dynamical systems. We offer several examples, discuss various types of orbits, and representations of systems by directed graphs. Next, we recall  finite fields and polynomial dynamics over finite fields. We then prove a result about when a given system has no preperiodic orbits. Finally, we describe a handful of questions for further pursuit, and an appendix where we outline the (classical) existence and construction of finite fields.

\section{Dynamical Systems}

Dynamical systems are a fundamental discipline in mathematics.  Moreover, dynamical systems abound in the natural and social worlds.  For instance, dynamics are often applied to game/equilibrium theory, which itself is used in devising pricing schemes in economics.  Chemical and biological networks/pathways are often studied using dynamical systems as well. A standard course in dynamical systems is based on \cite{katok}. In what follows, we will only consider {\it discrete} dynamical systems, i.e., a set $S$ and a function $f: S \to S$.

To get a feel for dynamics, consider a group of six friends: Avery, Blake, Charlie, Dakota, Emerson, and Finley.  Each person has a ball and follows a simple rule for to whom  to pass their ball:

\flushcolumns{
  \begin{multicols}{2}
  \begin{itemize}
\item Avery always passes to Blake;
\item Blake always passes to Avery;
\item Charlie always passes to Dakota;

\columnbreak

\item Dakota always passes to Emerson;
\item Emerson always passes to Charlie; 
\item Finley always holds their ball.
\end{itemize}
\end{multicols}}

\noindent
Pictorially the situation looks as follows.\\

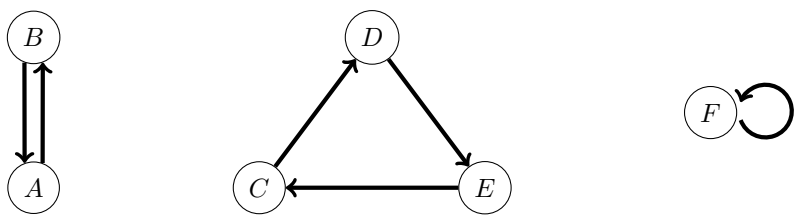
\begin{center}        
        \begin{tikzpicture}

\node[circle] (a) at (0,0) [draw] {$A$};
\node[circle] (b) at (0,2) [draw] {$B$};
\node[circle] (c) at (3,0) [draw] {$C$};
\node[circle] (e) at (6,0) [draw] {$E$};
\node[circle] (d) at (4.5,2) [draw] {$D$};
\node[circle] (f) at (9,1) [draw] {$F$};
\draw[ultra thick,->] (a.70) --  (b.-70);
\draw[ultra thick,->] (b.-110) --(a.110);
\draw[ultra thick, ->] (9.4,0.9) arc (-160:160:0.35cm);
\draw[ultra thick,->] (c) --  (d);
\draw[ultra thick,->] (d) --  (e);
\draw[ultra thick,->] (e) --  (c);
        \end{tikzpicture}
        \end{center}

\noindent
Observe that there are different types of patterns (orbits) in the picture, e.g., it takes three rounds for the Charlie's ball to return to them, while it only takes two rounds for Avery's ball to return to Avery, Finley's ball doesn't move at all.

What we have just described is a (discrete) {\it dynamical system}, more precisely we had a set $S$---in the example, let us represent the friends by their first initial, so $S = \{A, B, C, D, E, F\}$---and a function $ f: S \to S$.  We then considered the iterates of $f$, i.e., the $n$-fold compositions of $f$, $f^n = f \circ f \circ \dotsb \circ f$. In the friends example, $f$ was the rule for passing, e.g., $f$ maps Avery to Blake, so $f(A)= B$.  We then considered the {\it forward orbits} of various balls, e.g., Charlie's ball went to Dakota, then onto Emerson, then returned to Charlie. In terms of the function, we have $f^3 (C) = C$, and we say that Charlie (or their ball) is a {\it periodic point} of {\it period} 3; similarly, Dakota and Emerson are periodic points of period 3.  In addition, Avery and Blake are both periodic points with period 2. Finley has period 1, so they are called a {\it fixed point}. We will see below another orbit type: a {\it pre-periodic point}, i.e., an element $p$ of our set $S$ which becomes periodic after some number of iterations.

\begin{wrapfigure}{r}{0.4\textwidth}
\begin{center}        
        \begin{tikzpicture}

\node[circle] (0) at (0,0) [draw] {$1$};
\node[circle] (1) at (2,0) [draw] {$2$};
\node[circle] (2) at (4,0) [draw] {$3$};
\node[circle] (3) at (0,2) [draw] {$4$};
\node[circle] (4) at (2,2) [draw] {$5$};
\node[circle] (5) at (4,2) [draw] {$6$};
\node[circle] (6) at (0,4) [draw] {$7$};
\node[circle] (7) at (2,4) [draw] {$8$};
\node[circle] (8) at (4,4) [draw] {$9$};
\draw[ultra thick,->] (2) --  (5);
\draw[ultra thick,->] (1) --  (2);
\draw[ultra thick, ->] (0.175,.35) arc (-60:240:0.35cm);
\draw[ultra thick,->] (5) --  (8);
\draw[ultra thick,->] (4) --  (5);
\draw[ultra thick,->] (8) --  (7);
\draw[ultra thick,->] (7) --  (4);
\draw[ultra thick,->] (3.115) --  (6.-115);
\draw[ultra thick,->] (6.-65) --  (3.65);
        \end{tikzpicture}
        \end{center}
        \vspace{1ex}
\end{wrapfigure}

When the underlying set, $S$, of our dynamical system, $(S, f)$, has only a few elements, it is convenient to describe our system as a {\it directed graph}.  As an example, let $S = \{1,2, \dotsc, 9\}$ be the first nine natural numbers. Let us arrange $S$ as a square grid and consider the dynamical system described by the graph at right.

Representing the system as a directed graph makes it convenient to identify the orbit types.  Indeed, we see that the element `1' is a fixed point; there are two periodic orbits/cycles: $(4,7)$ and $(5,6,9,8)$ of period two and four respectively.   Additionally,  both `2' and `3' are preperiodic points as they eventually fall into the $(5,6,9,8)$ periodic orbit.

\section{Finite Fields: Polynomial Dynamics}

The mathematical notion of a {\it field} is as a place to do arithmetic. We recall the definition of field  below; a particular consequence is that the product of two nonzero elements is again nonzero. 

\begin{definition}
A {\it field} is a set $F$ equipped with two associative and commutative binary operations $+$ and $\times$, such that
\begin{itemize}
\item There exist elements $0$ and $1$ in $F$ which are the identity elements for $+$ and $\times$ respectively;
\item Every element $a$ of $F$ has an additive inverse: an element $-a$ such that 
$a+(-a)=0$;
\item Every nonzero element $a$ of $F$ has a multiplicative inverse: an element $a^{-1}$ such that $a \times a^{-1} =1$;
\item The operation $\times$ distributes over $+$: 
$a \times (b +c) = a \times b + a \times c$.
\end{itemize}

\end{definition}

Several fields are familiar to us, e.g., the real numbers $\RR$ or the complex numbers $\CC$. After some thought, one can see that the rational numbers (fractions of integers with nonzero denominator) also has the structure of a field.  In contrast, note that the integers do not form a field: given a nonzero integer other than 1 or -1, say 5, there does not exist another integer with which the resulting product is equal to 1: $1/5$ is a rational number, but not an integer!

The fields with which we are most familiar are infinite, however, there exist fields with only finitely many elements as well.  Consider the integers $\{0,1,2,3,4\}$ equipped with the standard operation of addition, $+$, and multiplication, $\times$, except that we reduce modulo division by 5; let us label these new operations as $\widetilde{+}$ and $\widetilde{\times}$. That is, given two such integers, say 2 and 4, we consider their standard product, but then only remember the remainder after dividing by 5.  More explicitly, $2 \times 4 =8$, 8 divided by $5$ is 1 with remainder 3, so $2 \widetilde{\times} 4 = 3$. (Verify that $2 \widetilde{+} 4 = 1$.) Equipped with the operations $\widetilde{+}$ and $\widetilde{\times}$, the set $\{0,1,2,3,4\}$ defines a field, we will denote this field $\FF_5$. More generally, given any prime number $p$, there is a field $\FF_p$ with underlying set $\{0,1,2,\dotsc, p-2, p-1\}$, as we can do arithmetic modulo any prime $p$!

More generally, given any power of a prime, there exists a corresponding finite field.  That is, let $n$ be a natural number and $p$ a prime number, there exists a unique field with $p^n$ elements, we denote this field $\FF_{p^n}$, e.g., $\FF_8$ for $n=3$ and $p=2$ or $\FF_9$ for $n=2$ and $p=3$. We outline the existence and uniqueness of the field $\FF_{p^n}$ at the end of the article, but let us describe explicitly our examples of $\FF_8$ and $\FF_9$.

\begin{example}
We present the field $\FF_8$ as certain collection of polynomials.
As a set, let 
\[
\FF_8=\{0,1,x,x+1,x^2, x^2+1,  x^2+x, x^2+x+1\}.
\]
The operations $+$ and $\times$ are the standard sum and multiplication of polynomials, however, we must reduce all coefficients modulo 2, and use the replacement rule which sends $x^3$ to $x+1$. For instance, we compute the following
\[
\begin{array}{lll}
(x+1) \times (x^2+x+1) &=& x^3+x^2+x+x^2+x+1\\
&=& x^3+2x^2+2x+1\\
&=& x+1+2x^2+2x+1\\
&=& 2x^2+3x+2\\
&=& 0+x+0\\
&=&x
\end{array}
\]
\end{example}

\begin{example}
The field $\FF_9$ has underlying set $\{0,1,2, x, x+1, x+2, 2x, 2x+1, 2x+2\}$.  The operations $+$ and $\times$ are again those of polynomials.  We now reduce all coefficients modulo 3 and use the replacement rule that $x^2 \equiv 2x+1$.
\end{example}

It is natural to wonder if there is a field of order $n$ for any natural number $n$...the answer is a resounding NO!  Consider $n=6$, it is the case that the structure with respect to the operation $+$ is that of addition modulo 6.  One then sees that $2\times 3 \equiv 0$, so there are nontrivial elements whose product is zero, which would contradict our earlier observation about fields.

Given an (algebraic) equation, one could ask if it has solutions in a given field.  For instance, the equation $x^4-1=0$ has two solutions in $\RR$, but four solutions in $\CC$. Note that $x^4-1=0$  has a unique solution over $\FF_8$, while there are four solutions over $\FF_9$: 1, 2, $x+2$, and $2x+1$.

Counting the number of solutions of equations over various fields is a basic motivation for much of number theory and algebraic geometry.  There are important functions which count the number of solutions of an equation over various fields, these are called {\it Zeta functions}.  Zeta functions occupy a central role in many important theorems and conjectures: Hasse's Theorem, the Weil Conjectures, the Riemann Hypothesis, etc; we offer some explicit questions in Section \ref{sect:q}.

\subsection{Polynomial Dynamics over Finite Fields}

We now come to the main focus of this note: dynamical systems determined by polynomial functions.  Our dynamical system will have as underlying set the elements of a finite field $\FF_{p^n}$.  Given a polynomial, $f(t)$, with integer coefficients, we can interpret $f$ as a map from $\FF_{p^n}$ to itself. More explicitly, given an element $\alpha$ of $\FF_{p^n}$ we can evaluate $f(\alpha)$, apply the replacement rule defining $\FF_{p^n}$, and then reduce all coefficients modulo $p$. It is more concrete to consider several examples, as we now do.

Consider the polynomial $f(t) = t^2+1$ over the field with nine elements, $\FF_9$. The  directed graph below describes the dynamical system generated by iterating $f(t)$.

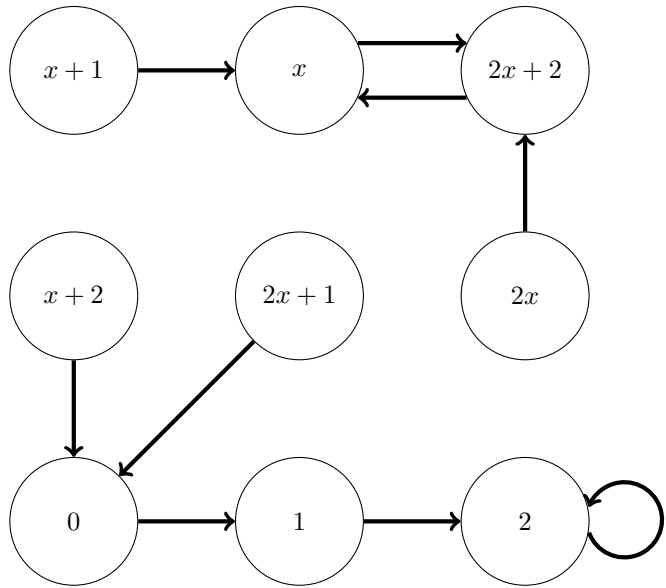
\begin{center}        
        \begin{tikzpicture}

\node[circle,minimum size=1.7cm] (0) at (0,0) [draw] {$0$};
\node[circle,minimum size=1.7cm] (1) at (3,0) [draw] {$1$};
\node[circle,minimum size=1.7cm] (2) at (6,0) [draw] {$2$};
\node[circle,minimum size=1.7cm] (3) at (0,3) [draw] {$x+2$};
\node[circle,minimum size=1.7cm] (4) at (3,3) [draw] {$2x+1$};
\node[circle,minimum size=1.7cm] (5) at (6,3) [draw] {$2x$};
\node[circle,minimum size=1.7cm] (6) at (0,6) [draw] {$x+1$};
\node[circle,minimum size=1.7cm] (7) at (3,6) [draw] {$x$};
\node[circle,minimum size=1.7cm] (8) at (6,6) [draw] {$2x+2$};
\draw[ultra thick,->] (0) --  (1);
\draw[ultra thick,->] (1) --  (2);
\draw[ultra thick, ->] (6.85,-.15) arc (-160:160:0.5cm);
\draw[ultra thick,->] (3) --  (0);
\draw[ultra thick,->] (4) --  (0);
\draw[ultra thick,->] (5) --  (8);
\draw[ultra thick,->] (6) --  (7);
\draw[ultra thick,->] (7.25) --  (8.155);
\draw[ultra thick,->] (8.-155) --  (7.-25);
        \end{tikzpicture}
        \end{center}

\noindent
This dynamical system has one fixed point, a cycle of length two, and six preperiodic points.

We now work with the same polynomial, $f(t)=t^2+1$, but over the field $\FF_8$.

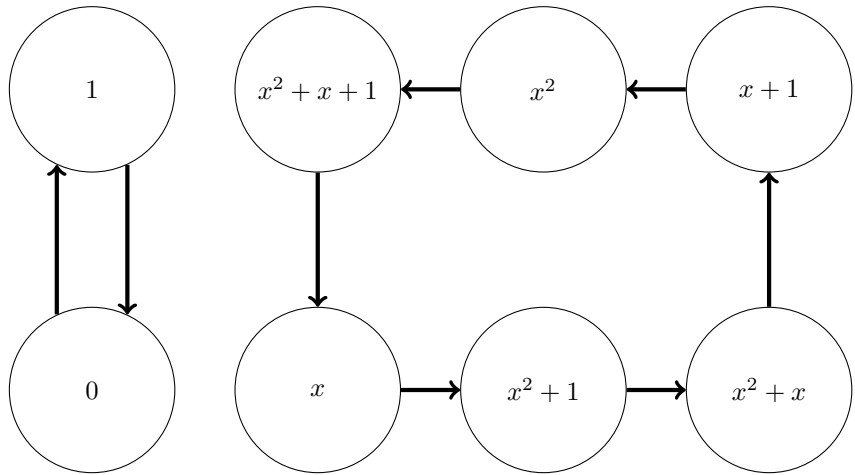
\begin{center} 
\vspace{1ex}       
        \begin{tikzpicture}

\node[circle,minimum size=2.2cm] (0) at (0,0) [draw] {$0$};
\node[circle,minimum size=2.2cm] (1) at (3,0) [draw] {$x$};
\node[circle,minimum size=2.2cm] (2) at (6,0) [draw] {$x^2+1$};
\node[circle,minimum size=2.2cm] (3) at (9,0) [draw] {$x^2+x$};
\node[circle,minimum size=2.2cm] (4) at (0,4) [draw] {$1$};
\node[circle,minimum size=2.2cm] (5) at (3,4) [draw] {$x^2+x+1$};
\node[circle,minimum size=2.2cm] (6) at (6,4) [draw] {$x^2$};
\node[circle,minimum size=2.2cm] (7) at (9,4) [draw] {$x+1$};
\draw[ultra thick,->] (0.115) --  (4.245);
\draw[ultra thick,->] (4.295) --  (0.65);
\draw[ultra thick,->] (1) --  (2);
\draw[ultra thick,->] (2) --  (3);
\draw[ultra thick,->] (3) --  (7);
\draw[ultra thick,->] (7) --  (6);
\draw[ultra thick,->] (6) --  (5);
\draw[ultra thick,->] (5) --  (1);

        \end{tikzpicture}
        \vspace{1ex}
                \end{center}

\noindent
Note that there are no preperiodic points, no fixed points, and each element is contained in one of two periodic orbits! Hence, it is clear that the dynamics of iterating a polynomial depends significantly on the field. In the second example, the lack of preperiodic points is no coincidence, but rather a consequence of some elementary number theory as we will see in the next section.

\section{Guaranteeing no preperiodic points}

In this section we prove an elementary, but useful result about polynomial dynamics over finite fields. The key technical tool is the following result, typically attributed to Fermat and its extension that the Frobenius map is an automorphism of a finite field.

\begin{fermat}
Let $p$ be a prime.  Then for any integer $a$,
\[
a^p \equiv a \pmod{p}.
\]
\end{fermat}

Fermat's result is incredibly useful, and it's refinement due to Euler, is one of the key ingredients of the RSA (Rivest-Shamir-Adleman) encryption scheme (also known as public-key cryptography). Though first described in 1977, the RSA algorithm, with some technical advances and {\it padding}, remains the most widely used encryption scheme for secure communication and digital commerce in the world.

There is an extension of Fermat's Little Theorem which will prove useful. To begin, given a finite field $\FF_{p^n}$, there is a natural field homomorphism, the $p$th power map (also called the {\it Frobenius endomorphism}):
\[
\phi : \FF_{p^n} \to \FF_{p^n}, \quad a \mapsto a^p.
\]
That the map $\phi$ is actually a field homomorphism follows from the computations
\[
(xy)^p = x^p y^p \quad \quad \text{ and } \quad \quad (x+y)^p = x^p +y^p,
\]
with the latter a consequence of the Binomial Theorem.

Before proving that the Frobenius map is actually an automorphism, we need a standard result about field homomorphisms.

\begin{lemma}
Let $\psi : \GG \to \HH$ be a field homomorphism.  Then $\psi$ is injective.
\end{lemma}

\begin{proof}
Let $x,y \in \GG$ such that $f(x)=f(y)$. Define an element $d=x-y$, so 
\[
f(d) = f(x-y) = f(x)-f(y) =0.
\]
We will show $d=0$. By way of contradiction, suppose not. So $d \neq 0$, then
\[
1=f(1) = f(d d^{-1}) = f(d) f(d^{-1}) = 0 f(d^{-1}) =0,
\]
which is a contradiction in the field $\HH$.
\end{proof}

\begin{proposition}
The Frobenius endomorphism $\phi : \FF_{p^n} \to \FF_{p^n}$ is an automorphism, in particular, $\phi$ is injective.
\end{proposition}

\begin{proof}
Since $\FF_{p^n}$ is finite, any injective map will automatically be surjective, hence its enough to prove that $\phi$ is injective. That $\phi$ is injective is the content of the previous lemma.
\end{proof}

\begin{theorem}
Let $p$ be a prime number, $m$ a natural number, and $c$ an element of $\FF_p$.  For all natural numbers $n$, the dynamical system generated by $f(t) = t^{p^m} +c$ has no preperiodic points over $\FF_{p^n}$.
\end{theorem}

\begin{proof}
Note that our finite system will have no preperiodic points when the map $f$ is injective, hence we prove injectivity of $f$. Suppose $a,b \in \FF_{p^n}$ such that $a^{p^m}+c = b^{p^m}+c$. By cancellation with respect to $+$, we then know that $a^{p^m} = b^{p^m}$.

We now induct on $m$. In the case $m=1$, we deduce that $a=b$ by the proceeding proposition.
Now assume that for $1 \le k \le N$, if $a^{p^N} = b^{p^N}$, then $a=b$ in $\FF_{p^n}$. We consider the case where $k=N+1$.  Suppose, that $a^{p^{N+1}}=b^{p^{N+1}}$.  Note that, $a^{p^{N+1}}=a^p a^{p^N}$ and similarly, $b^{p^{N+1}}=b^p b^{N+1}$. Applying the inductive hypothesis, to $a^{p^N}$ and $b^{p^N}$, we conclude $a^p = b^p$ which is the base case of our induction.  Therefore, for all natural numbers $m$, we have proven that if $a^{p^m} = b^{p^m}$, then $a=b$ and we have established the injectivity of $f$.
\end{proof}

Note that the lack of pre-periodic points does not imply that our system is of the form of the proposition. Indeed, over $\FF_2$ the polynomial $f(z) = z^3$ only has fixed points.

%
%

\section{A Dynamical Proof of Fermat's Little Theorem}

There are many proofs of Fermat's Little Theorem and nearly every text on number theory contains at least one such proof.  Euler gave a proof of the theorem based on the Binomial Theorem; there are also proofs that use basic group theory and a proof due to Ivory and Dirichlet that uses modular arithmetic.  Motivated by the connection to dynamical systems, we recall a proof of the theorem due to Iga \cite{iga} (see also \cite{dragovic}).

Let $n$ be an integer, Iga's idea is to study the dynamics on the interval $[0,1]$ of the following function:
\[
T_n (x) = \left \{ \begin{array}{rr} nx - \lfloor nx \rfloor ,& x \neq 1\\ 1,& x =1 \end{array} \right.
\]
Note that $y-\lfloor y \rfloor$ is just the factional part of the real number $y$.

The reader is encouraged to draw a graph of the function $T_n (x)$, which will make the following clear.

\begin{proposition}
Let $\ell,m,n$ be integers with $n>1$. Then,
\begin{itemize}
\item[(a)] $T_n (x)$ has $n$ fixed points;
\item[(b)] $(T_m \circ T_\ell) (x) = T_{m \ell} (x)$.
\end{itemize}
\end{proposition}

Let $a$ be an integer and $n$ a prime number.  We will now use the dynamics determined by $T_a$ to show that $p$ divides the difference $a^p-a$, which is equivalent to Fermat's Little Theorem.

Indeed, let $\mathcal{P}$ denote the set of points which have period $p$ with respect to iterating $T_a$.  Equivalently, $\mathcal{P}$ is the set of fixed points of $T_{a^p}$, so $\lvert \mathcal{P} \rvert =a^p$.  Of these $a^p$ fixed points, $a$ are already fixed by $T_a$. As $p$ is prime, we know that $a^p-a$ points then have minimal period $p$, and hence the number of orbits of length $p$ is exactly $(a^p-a)/p$.  But then, as a finite count of orbits,  $(a^p-a)/p$ must be an integer; that is $p$ divides $a^p-a$.

\section{Further Questions}\label{sect:q}

Elliptic curves are geometric objects that can be defined over arbitrary fields.  An elliptic curve, $\mathcal{E}$, defined over the complex numbers $\mathbb{C}$ has a particularly nice description: $\mathcal{E}$ is a quotient of $\mathbb{C}$ by a lattice $L=\mathbb{Z} \oplus \mathbb{Z} \tau$ for some complex number $\tau$ in the upper half plane. This description makes it clear that an elliptic curve over $\CC$ is, as a space, simply a torus.  Elliptic curves are also examples of {\it algebraic varieties}, in that they are the space of solutions to an algebraic equation.  Indeed, our curve $\mathcal{E}$ can be described as the space of solutions to a cubic equation in two variables:
\[
y^2 +a_1 xy + a_2 y = x^3 + a_3 x^2 + a_4 x + a_5.
\]
The equation above is called the (generalized) Weierstrass equation.  Given an elliptic curve, $\mathcal{E}$, we can ask how many points this curve has over the finite field $\FF_q$. Equivalently, this is the number of solutions to the corresponding Weierstrass equation in $\FF_q$; denote this number $N(q)$. It is a classical theorem of Hasse that
\[
\lvert N(q)-q-1) \rvert \le 2 \sqrt{q}.
\]

The collection of periodic points in a dynamical system determined by iterating a polynomial (even rational) map also has the structure of an algebraic variety, see \cite{SilvBook}.  This description was extended in \cite{Gao} to the case of pre-periodic points. Hence, one could ask the following.

\begin{itemize}
\item[(Q1)] Can one bound the number of pre-periodic points in analogy with Hasse's Theorem?
\end{itemize}

For a fixed prime $p$, one can assemble the numbers $N(p^m)$ into a formal power series (with further coefficients), called the (local) {\it zeta function}. (Again, see \cite{SilvBook} for more.) For algebraic curves, this power series is analogous to the Riemann Zeta Function which is the subject of the famous Riemann Hypothesis.  The nature of the zeta function of an algebraic curve led to some of the deepest and encompassing mathematics of the 20th century: The Weil Conjectures.

Pre-periodic zeta functions have only been lightly studied, e.g., \cite{Waddington} and there are many aspects to be explored.

\begin{itemize}
\item[(Q2)] What can one determine about the pre-periodic zeta function?
\end{itemize}

By ``pasting together" the local zeta functions over each prime, one forms the Dedekind Zeta Function. In good situations, one can read off a group of symmetries from the Dedekind Zeta Function: the Galois group. The case of periodic points has been studied in \cite{Morton}.

\begin{itemize}
\item[(Q3)] Can one compute an analogue of the Galois group for pre-periodic varieties?
\end{itemize}

\section*{Appendix: Existence and Uniqueness of Finite Fields}

We now prove that given a prime number $p$ and a natural number $n$, up to isomorphism, there is a unique field with $p^n$ elements. Hence, our notation $\FF_{p^n}$ is unambiguous (up to isomorphism). The proof of the theorem uses some graduate level abstract algebra that may not be familiar to many readers, a standard reference is \cite{algebra}.

\begin{theorem}
Let $p$ be a prime and $n$ a natural number. Then,
\begin{itemize}
\item (Existence) There exists a field with $p^n$ elements.
\item (Uniqueness) All fields with $p^n$ elements are isomorphic.
\end{itemize}
\end{theorem}

\begin{proof}[Proof of Existence]
Consider the ring of polynomials with integer coefficients $\ZZ[x]$.  Let $p(x)$ be an irreducible polynomial of degree $n$.  As $\ZZ$ is a principal ideal domain, the ideal generated by $p(x)$ in $\ZZ[x]$ is maximal, so $\ZZ[x]/(p(x))$ is a field.  It is straightforward to verify that $\ZZ[x]/(p(x))$ has $p^n$ elements.
\end{proof}

\begin{proof}[Sketch of Proof of Uniqueness]
The main idea is to consider the splitting field $\mathbb{K}$ of the polynomial $x^{p^n}-x$ over $\FF_p$.  This splitting field has ${p^n}$ elements by minimality.  Now the group of units $\FF_{p^n}^\times$ is a cyclic group of order $p^n-1$ and each element is a root of $x^{p^n-1} -1$. Hence, $\FF_{p^n}$ is a splitting field of $x^{p^n}-x$ over $\FF_p$.  By uniqueness of splitting fields, we have that $\FF_{p^n} \cong K$.
\end{proof}

\begin{theorem}
Let $n \in \NN$ such that there exist distinct primes $p$ and $q$ both dividing $n$. There exists no field of order $n$.
\end{theorem}

The theorem follows from a simpler proposition.

\begin{proposition}
Let $p$ and $q$ be distinct primes, then there exists no field of order $pq$.
\end{proposition}

\begin{proof}
Suppose $\FF$ is a field of order $pq$ and let 1 denote the unit with respect to multiplication $\times$. One can show that the underlying abelian group $(\FF, +,0)$ is isomorphic to the cyclic group $\ZZ/pq$ (an application of the Sylow Theorems).  Let $\rho = p \times 1$ and $\delta = q \times 1$. Then,
\[
\rho \times \delta = pq \times 1 = 0
\]
in $\FF$. That is, $\rho$ (and $\delta$) is a nontrivial zero divisor, contradicting $\FF$ being a field.
\end{proof}

\section*{About the authors:}

\subsection*{Ryan Grady}
Ryan Grady is an Assistant Professor of Mathematics and Director of the Directed Reading Program in Mathematical Sciences at Montana State.

Department of Mathematical Sciences,
Montana State University, Bozeman 59717.
ryan.grady1@montana.edu

\subsection*{Mark Poston}
Mark Poston is a recent Bachelor of Science in Mathematics (with honors) degree recipient and a current Master's of Science candidate at Montana State. He is a member of Pi Mu Epsilon and was active in math club and math modeling competitions as an undergraduate student.

Department of Mathematical Sciences,
  Montana State University, Bozeman 59717.

\end{document}